\documentclass[a4paper]{amsart}

\usepackage{amssymb, enumerate}
\usepackage{xy} \xyoption{all}
\usepackage{aliascnt,hyperref}

\newtheorem{lma}{Lemma}[section]

\newaliascnt{thmCt}{lma}
\newtheorem{thm}[thmCt]{Theorem}
\aliascntresetthe{thmCt}

\newaliascnt{corCt}{lma}
\newtheorem{cor}[corCt]{Corollary}
\aliascntresetthe{corCt}

\newaliascnt{prpCt}{lma}
\newtheorem{prp}[prpCt]{Proposition}
\aliascntresetthe{prpCt}

\newtheorem*{thm*}{Theorem}

\theoremstyle{definition}

\newaliascnt{pgrCt}{lma}
\newtheorem{pgr}[pgrCt]{}
\aliascntresetthe{pgrCt}

\newaliascnt{dfnCt}{lma}

\aliascntresetthe{dfnCt}

\newaliascnt{rmkCt}{lma}
\newtheorem{rmk}[rmkCt]{Remark}
\aliascntresetthe{rmkCt}

\newaliascnt{rmksCt}{lma}
\newtheorem{rmks}[rmksCt]{Remarks}
\aliascntresetthe{rmksCt}

\newaliascnt{exaCt}{lma}
\newtheorem{exa}[exaCt]{Example}
\aliascntresetthe{exaCt}

\newaliascnt{qstCt}{lma}
\newtheorem{qst}[qstCt]{Question}
\aliascntresetthe{qstCt}

\DeclareMathOperator{\id}{id}
\DeclareMathOperator{\ev}{ev}
\DeclareMathOperator{\Sh}{Sh}

\newcommand{\andSep}{\,\,\,\text{ and }\,\,\,}
\newcommand{\ca}{{\mbox{$C^*$-al}\-ge\-bra}}
\newcommand{\Bdd}{\mathcal{B}}
\newcommand{\stHom}{${}^*$-homomorphism}
\newcommand{\CAlgebra}[2]{C^\ast\big\langle #1\ |\ #2\big\rangle}
\newcommand{\CAlgebraUnital}[2]{C^\ast_1\big\langle #1\ |\ #2\big\rangle}
\newcommand{\CAlgebraBig}[2]{C^\ast\left\langle{ \begin{array}{c} #1 \end{array} {\left| \begin{array}{l} #2 \end{array} \right\rangle}} \right.}

\newcommand{\vect}[1]{\mathbf{#1}}
\newcommand{\KK}{{\mathbb{K}}}
\newcommand{\RR}{{\mathbb{R}}}
\newcommand{\NN}{{\mathbb{N}}}
\newcommand{\ZZ}{{\mathbb{Z}}}
\newcommand{\QQ}{{\mathbb{Q}}}
\newcommand{\CC}{{\mathbb{C}}}
\newcommand{\FF}{{\mathbb{F}}}

\title[Limits of semiprojectives]{Inductive limits of semiprojective $C^*$-algebras}
\author{Hannes Thiel}
\address{Hannes Thiel
Mathematisches Institut, Fachbereich Mathematik und Informatik der
Universit\"at M\"unster, Einsteinstrasse 62, 48149 M\"unster, Germany.}
\email{hannes.thiel@uni-muenster.de}
\urladdr{www.math.uni-muenster.de/u/hannes.thiel/}
\thanks{ The author was partially supported by the Deutsche Forschungsgemeinschaft (SFB 878 Groups, Geometry \& Actions).}
\subjclass[2010]%
{Primary
46L05, 
46M10; 
Secondary
46L85, 
46M20, 
54C55, 
54C56, 
55M15, 
55P55
}

\date{\today}

\begin{document}

\begin{abstract}
We prove closure properties for the class of \ca{s} that are inductive limits of semiprojective \ca{s}.
Most importantly, we show that this class is closed under shape domination, and so in particular under shape and homotopy equivalence.
It follows that the considered class is quite large.
It contains for instance the stable suspension of any nuclear \ca{} satisfying the UCT and with torsion-free $K_0$-group.
In particular, the stabilized \ca{} of continuous functions on the pointed sphere is isomorphic to an inductive limit of semiprojectives.
\end{abstract}

\maketitle

\section{Introduction}

Shape theory is a tool to study global properties of metric spaces that have singularities.
This is done by approximating the space under consideration by nicer spaces without singularities.
Properties of the original space are encoded in the approximating system.

This idea has been transferred to the study of \ca{s} by Blackadar in \cite{Bla85ShapeThy}.
The building blocks of this noncommutative shape theory are the semiprojective \ca{s};
see \autoref{sec:shape} for definitions.
One therefore seeks to approximate a given \ca{} $A$ by semiprojective \ca{s}.
This leads to the following fundamental question:

\begin{qst}[{Blackadar, \cite[4.4]{Bla85ShapeThy}}]
\label{qst:sss}
Is every separable \ca{s} isomorphic to a sequential inductive limit of separable, semiprojective \ca{s}?
\end{qst}

The analogue of this question for metric space has a positive answer:
Every metric space is homeomorphic to an inverse limit of absolute neighborhood retracts.

Since the answer to \autoref{qst:sss} is unknown, Blackadar developed a more general notion of approximation, called a \emph{shape system};
see \autoref{sec:shape} for details.
Every separable \ca{} is isomorphic to the inductive limit of a shape system.

In contrast, we say that a separable \ca{} has a \emph{strong shape system} if it is isomorphic to a sequential inductive limit of separable, semiprojective \ca{s}.
Thus, \autoref{qst:sss} asks if every separable \ca{} has a strong shape system.
The main result of this paper is:


\begin{thm*}[{\ref{prp:sss_shDomination}}]
Let $A$ and $B$ be separable \ca{s} such that $A$ is shape dominated by $B$.
Then, if $B$ has a strong shape system, so does $A$.
\end{thm*}

Consequently, if two separable \ca{s} are shape equivalent (in particular, if they are homotopy equivalent), then one has a strong shape system if and only if the other does;
see \autoref{prp:sss_shEquivalence}.
In \autoref{prp:sss_permanence}, we summarize closure properties for the class of separable \ca{s} that have a strong shape system.

We apply \autoref{prp:sss_shDomination} to show that many nuclear \ca{s} are inductive limits of semiprojective \ca{s}.

\begin{thm*}[{\ref{prp:ssa_for_HS}}]
Let $A$ be a separable, stable, nuclear, homotopy symmetric \ca{} satisfying the UCT.
Assume that $K_0(A)$ is torsion-free.
Then $A$ has a strong shape system.
\end{thm*}

In particular, if $A$ is a separable, nuclear \ca{} satisfying the UCT, then the stable suspension $\Sigma A\otimes\KK$ has a strong shape system if $K_1(A)$ is torsion-free, and the stable second suspension $\Sigma^2 A\otimes\KK$ has a strong shape system if $K_0(A)$ is torsion-free;
see \autoref{prp:sss_suspension_nuclear}.

Our results provide a partial answer to \autoref{qst:sss} by substantially increasing the class of \ca{s} that are known to have a strong shape system.
Moreover, we obtain many new concrete examples of \ca{s} that are inductive limits of semiprojective \ca{s}.
For example, let $A$ be a UCT-Kirchberg algebra.
Then $C([0,1]^n,A)$ has a strong shape system for every $n\in\NN$;
see \autoref{exa:ShKirchberg}.
Further, if $X$ is a connected, compact, metrizable space and $x\in X$, then $C_0(X\setminus\{x\})\otimes A\otimes\KK$ has a strong shape system whenever $K_0(C_0(X\setminus\{x\})\otimes A)$ is torsion-free;
see \autoref{exa:sss_HS_Kirchberg}.

This paper proceeds as follows:
In \autoref{sec:shape}, we recall the basic notions of noncommutative shape theory.
In \autoref{sec:mappingCyl}, we generalize results of Loring and Shulman about cones of \ca{s} \cite[Section~7]{LorShu12NCSemialgLifting} to the setting of mapping cylinders.
We show that the mapping cylinder has a strong shape system if the domain of the defining morphism is semiprojective;
see \autoref{prp:mapping_cylinder_limit_SP}.

In \autoref{sec:sss}, we study closure properties of the class of \ca{s} that have a strong shape system.
We use the technical result about strong shape systems for mapping cylinders (\autoref{prp:mapping_cylinder_limit_SP}) to prove the main result \autoref{prp:sss_shDomination}.
In \autoref{sec:nuclearSSS}, we derive results about strong shape systems for nuclear \ca{s}.

\section*{Acknowledgements}

I am grateful to Dominic Enders for valuable feedback on a draft of this paper.

\section{Shape theory}
\label{sec:shape}

In this section, we recall the basic notions of shape theory for separable \ca{s} as developed by Blackadar in \cite{Bla85ShapeThy}.

We use the symbol $\simeq$ to denote homotopy equivalence.
By $A,B,C,D$ we usually denote \ca{s}.
A morphism between \ca{s} is understood to be a \stHom.
By ideals in a \ca{} we always mean closed, two-sided ideals.

A morphism $\varphi\colon A\to B$ is said to be \emph{projective} if for every \ca{} $C$, every ideal $J\lhd C$, and every morphism $\sigma\colon B\to C/J$, there exists a morphism $\psi\colon A\to C$ such that $\pi\circ\psi=\sigma\circ\varphi$, where $\pi\colon C\to C/J$ is the quotient morphism.
This is indicated in the commutative diagram on the left below.

A morphism $\varphi\colon A\to B$ is said to be \emph{semiprojective} if for every \ca{} $C$, every increasing sequence $J_0\subseteq J_1\subseteq\ldots$ of ideals in $C$, and for every morphism $\sigma\colon B\to C/\overline{\bigcup_nJ_n}$, there exist some $n\in\NN$ and a morphism $\psi\colon A\to C/J_n$ such that $\pi_{\infty,n}\circ\psi=\sigma\circ\varphi$, where $\pi_{\infty,n}\colon C/J_n\to C/\overline{\bigcup_nJ_n}$ is the quotient morphism.
This is indicated in the commutative diagram on the right below.

\begin{figure}[h]
\centering
\begin{minipage}{.5\textwidth}
\centering
\makebox{
\xymatrix{
& & C \ar[d]^{\pi} \\
A \ar[r]_{\varphi} \ar@{..>}[urr]^{\psi} & B \ar[r]_-{\sigma}& C/J
}}
\end{minipage}%
\begin{minipage}{.5\textwidth}
\centering
\makebox{
\xymatrix@R-5pt{
& & C \ar[d] \\
& & C/J_n \ar[d]^{\pi_{\infty,n}} \\
A \ar[r]_{\varphi} \ar@{..>}[urr]^{\psi} & B \ar[r]_-{\sigma}  & C/\overline{\bigcup_n J_n}
}}
\end{minipage}
\end{figure}
A \ca{} $A$ is \emph{(semi)projective} if the identity map $\id_A\colon A\to A$ is.

By a \emph{sequential inductive system} we mean a sequence $A_0,A_1,A_2,\ldots$ of \ca{s} together with morphisms $\gamma_{n+1,n}\colon A_n\to A_{n+1}$ for each $n\in\NN$.
Given $n,m\in\NN$ with $n<m$, we set $\gamma_{m,n}:=\gamma_{m,m-1}\circ\ldots\circ\gamma_{n+2,n+1}\circ\gamma_{n+1,n}\colon A_n\to A_m$.
We call $\gamma_{m,n}$ the connecting morphisms of the system.
We use $\varinjlim_n A_n$ to denote the inductive limit.
It comes with natural morphisms $\gamma_{\infty,n}\colon A_n\to\varinjlim_n A_n$ for each $n\in\NN$.

Let $A$ be a separable \ca.
A \emph{shape system} for $A$ is a sequential inductive system $(A_n,\gamma_{n+1,n})$ of separable \ca{s} such that $A\cong\varinjlim_n A_n$ and such that the connecting morphisms $\gamma_{n+1,n}$ are semiprojective.
By \cite[Theorem~4.3]{Bla85ShapeThy}, every separable \ca{} has a shape system..

Let $\mathcal{A}=(A_n,\gamma_{n+1,n})$ and $\mathcal{B}=(B_k,\theta_{k+1,k})$ be inductive systems.
Then $\mathcal{A}$ is said to be \emph{shape dominated} by $\mathcal{B}$, denoted $\mathcal{A}\precsim\mathcal{B}$, if there exist two strictly increasing sequences $n_0<n_1<\ldots$ and $k_0<k_1<\ldots$ in $\NN$, and morphisms $\alpha_i\colon A_{n_i}\to B_{k_i}$ and $\beta_i\colon B_{k_i}\to A_{n_{i+1}}$ such that $\beta_i\circ\alpha_i\simeq\gamma_{n_{i+1},n_i}$ for $i\in\NN$.
If also $\alpha_{i+1}\circ\beta_i\simeq\theta_{k_{i+1},k_i}$ for all $i\in\NN$, then $\mathcal{A}$ and $\mathcal{B}$ are said to be \emph{shape equivalent}, denoted $\mathcal{A}\sim\mathcal{B}$.
The situation is indicated in the following diagram:
\[
\xymatrix{
A_{n_0} \ar[rr]^{\gamma_{n_1,n_0}} \ar[dr]_{\alpha_0}
& & A_{n_1} \ar[rr]^{\gamma_{n_2,n_1}} \ar[dr]_{\alpha_1}
& & A_{n_2} \ar[dr]_{\alpha_2} & \ldots
\\
& B_{k_0} \ar[rr]_{\theta_{k_1,k_0}} \ar[ur]_{\beta_0}
& & B_{k_1} \ar[rr]_{\theta_{n_2,n_1}} \ar[ur]_{\beta_1}
& & B_{k_2} & \ldots \ .
}
\]

The relation $\precsim$ for sequential inductive systems is transitive, and $\sim$ is an equivalence relation;
see \cite[Definition~4.6]{Bla85ShapeThy}.

By \cite[Corollary~4.9]{Bla85ShapeThy}, any two shape systems of a \ca{} are shape equivalent.
Therefore, the following definition makes sense:
A \ca{} $A$ is said to be \emph{shape dominated} by a \ca{} $B$, denoted $A\precsim_{\Sh}B$, if we have $\mathcal{A}\precsim\mathcal{B}$ for some, or equivalently every, shape system $\mathcal{A}$ for $A$ and $\mathcal{B}$ for $B$.
Similarly, $A$ is said to be \emph{shape equivalent} to $B$, denoted $A\sim_{\Sh}B$, if we have $\mathcal{A}\sim\mathcal{B}$ for shape systems $\mathcal{A}$ for $A$ and $\mathcal{B}$ for $B$.

Recall that $A$ is \emph{homotopy dominated} by $B$ if there exist morphisms $\alpha\colon A\to B$ and $\beta\colon B\to A$ with $\beta\circ\alpha\simeq\id_A$.
If also $\alpha\circ\beta\simeq\id_B$, then $A$ and $B$ are \emph{homotopy equivalent}.
By \cite[Corollary~4.11]{Bla85ShapeThy}, shape is coarser than homotopy:
If $A$ is homotopy dominated by (homotopy equivalent to) $B$, then $A\precsim_{\Sh} B$ ($A\sim_{\Sh}B$).

\section{Mapping cylinders}
\label{sec:mappingCyl}

In this section, we consider the mapping cylinder $Z_\varphi$ associated to a morphism $\varphi\colon A\to B$;
see \autoref{pgr:mappinCyl}.
We let $\varphi^+\colon A\to B^+$ denote the composition of $\varphi$ with the inclusion of $B$ into its forced unitization $B^+$.
Given a presentation of $A$ and $B$ with self-adjoint generators and relations, we show how to present $Z_{\varphi^+}$;
see \autoref{prp:presentation_mapping_cylinder}.
Let us assume that $A$ is semiprojective.
If we relax certain of the relations defining $Z_{\varphi^+}$, we obtain a semiprojective \ca{};
see \autoref{prp:softened_mapping_cylinder_SP}.
We deduce that $Z_{\varphi^+}$ is an inductive limit of semiprojective \ca{s};
see the proof of \autoref{prp:mapping_cylinder_limit_SP}.

To obtain the same conclusion for $Z_\varphi$, we use that an ideal $J$ in a semiprojective \ca{} $B$ is semiprojective if the quotient $B/J$ is projective;
see \cite[Corollary~3.1.3]{End16CharSemiprojSubhom}.
We deduce the main result of this section:
The mapping cylinder $Z_{\varphi}$ has a strong shape system;
see \autoref{prp:mapping_cylinder_limit_SP}.

\begin{pgr}
\label{pgr:presentation_ca}
The theory of universal \ca{s} given by generators and relations is very rich and closely connected to noncommutative shape theory.
Blackadar developed the basic theory in \cite[Section~1]{Bla85ShapeThy}.
A comprehensive study of \ca{} relations was presented by Loring \cite{Lor10CaRelations}.
We only recall the construction of a universal \ca{} given by self-adjoint generators subject to polynomial, order and norm relations.

Let $\vect{x}=(x_0,x_1,x_2,\ldots)$ be a sequence of (self-adjoint) generators.
We let $F(\vect{x})$ denote the free $\ast$-algebra on the self-adjoint generators $\vect{x}$.
Given a Hilbert space $H$, there is a natural bijective correspondence between sequences $\vect{X}=(X_k)_{k\in\NN}$ of self-adjoint elements in $\Bdd(H)$ and \stHom{s} $F(\vect{x})\to\Bdd(H)$.

Let $\vect{n}=(c_k)_{k\in\NN}$ be a sequence of numbers in $[0,\infty)$.
We say that a sequence $\vect{X}$ of self-adjoint elements in $\Bdd(H)$ satisfies the norm relations specified by $\vect{n}$ if $\|X_k\|\leq c_k$ for every $k\in\NN$.

A \emph{NC polynomial} $p$ in $\vect{x}$ is a polynomial in finitely many noncommuting variable from $\vect{x}$ with coefficients in $\CC$.
We always assume that NC polynomials have vanishing constant term.
An example of a NC polynomial is $x_1x_2^5x_1-3x_2x_3$.
Let $\vect{p}=(p_l)_{l\in\NN}$ be a sequence of NC polynomials in $\vect{x}$.
We say that a sequence $\vect{X}$ of self-adjoint elements in $\Bdd(H)$ satisfies the polynomial relations specified by $\vect{p}$ if $p_l(\vect{X})=0$ for every $l\in\NN$.

We also consider order-relations of the form $\alpha x_k\leq \beta x_l$, for some $\alpha,\beta\in\RR$ and $k,l\in\NN$.
We formalize this by letting $\vect{o}=((\alpha_m,\beta_m,s_{m},t_{m}))_{m\in\NN}$ be a sequence of tuples with $\alpha_m,\beta_m\in\RR$ and $s_{m},t_{m}\in\NN$.
We say that a sequence $\vect{X}$ of self-adjoint elements in $\Bdd(H)$ satisfies the order relations specified by $\vect{o}$ if $\alpha_m X_{s_{m}}\leq \beta_m X_{t_{m}}$ for every $m\in\NN$.

A \emph{representation} of $\langle\vect{x}|\vect{n},\vect{p},\vect{o}\rangle$ on a Hilbert space $H$ is a sequence $\vect{X}$ of self-adjoint operators in $\Bdd(H)$ satisfying the norm relations~$\vect{n}$, the polynomial relations~$\vect{p}$ and the order relations~$\vect{o}$.
Abusing notation, we identify a representation $\vect{X}$ of $\langle\vect{x}|\vect{n},\vect{p},\vect{o}\rangle$ on $H$ with the \stHom{} $\varphi\colon F(\vect{x})\to\Bdd(H)$ satisfying $\varphi(\vect{x})=\vect{X}$.
We obtain a universal $C^*$-seminorm on $F(\vect{x})$ given by
\[
\|z\| := \sup \big\{ \|\varphi(z)\| :
\varphi\colon F(\vect{x})\to\Bdd(H) \text{ representation of } \langle\vect{x}|\vect{n},\vect{p},\vect{o}\rangle \big\},
\]
for $z\in F(\vect{x})$.
The completion of $F(\vect{x})$ with respect to this $C^*$-seminorm is called the \emph{universal \ca{}} given by self-adjoint generators $\vect{x}$, subject to the relations specified by $\vect{n}$, $\vect{p}$, and $\vect{o}$.
Following \cite{LorShu12NCSemialgLifting} we denote it by
\[
\CAlgebraBig{
\vect{x}=(x_0,x_1,\ldots)
}{
-c_k\leq x_k \leq c_k, \quad k\in\NN \\
p_l(\vect{x})=0, \quad l\in\NN \\
\alpha_m x_{s_m}\leq \beta_m x_{t_m}, \quad m\in\NN
},
\]
or just $\CAlgebra{\vect{x}}{\vect{n},\vect{p},\vect{o}}$.

It was observed by Blackadar \cite[Example~1.3(b)]{Bla85ShapeThy} that every separable \ca{} has a presentation with countably many generators and corresponding norm relations and countably many polynomial relations.
Loring and Shulman \cite[Lemma~7.3]{LorShu12NCSemialgLifting} modified the construction of Blackadar to show that every separable \ca{} has such a presentation with \emph{self-adjoint} generators.
That is, given a separable \ca{} $A$, there is a sequence $\vect{x}=(x_k)_k$ of self-adjoint generators, a sequence $(c_k)_k$ of positive real numbers, and a sequence of NC polynomials $(p_l)_l$ in $\vect{x}$ such that
\begin{align*}
A\cong\CAlgebraBig{
\vect{x}=(x_0,x_1,\ldots)
}{
-c_k\leq x_k\leq c_k, \quad k\in\NN \\
p_l(\vect{x})=0, \quad l\in\NN
}.
\end{align*}

Assume that $A=\CAlgebra{\vect{x}}{\vect{n},\vect{p}}$, for a sequence of self-adjoint generators $\vect{x}$, and for some norm conditions $\vect{n}$ and NC polynomials~$\vect{p}$.
We set
\[
\CAlgebraUnital{\vect{x}}{\vect{n},\vect{p}}
:= \CAlgebraBig{
e,x_0,x_1,\ldots
}{
-c_k\leq x_k \leq c_k, \quad k\in\NN \\
p_l(\vect{x})=0, \quad l\in\NN \\
-1\leq e\leq 1, e^2=e \\
x_ke=x_k=ex_k,\quad k\in\NN
}.
\]
Then $\CAlgebraUnital{\vect{x}}{\vect{n},\vect{p}}$ is isomorphic to $A^+$, the forced unitization of $A$;
see \cite[II.1.2.1, p.53]{Bla06OpAlgs}.
For example, $\CC^+\cong\CC\oplus\CC$.
\end{pgr}

The following result is a slight generalization of \cite[Lemma~7.3]{LorShu12NCSemialgLifting} that will be used in \autoref{prp:presentation_mapping_cylinder}.

\begin{lma}
\label{prp:presentation-extending}
Let $B$ be a separable \ca{}, let $\bar{\vect{x}}=(\bar{x}_k)_k$ be a sequence in $B$, let $(c_k)_k$ be a sequence of positive real numbers, and let $(p_k)_k$ be a sequence of NC polynomials such that $-c_k\leq \bar{x}_k\leq c_k$ and $p_k(\bar{\vect{x}})=0$ for each $k$.

Then there exists a sequence $(d_j)_j$ of positive real numbers, and a sequence $(q_l)_l$ of NC polynomials in two sequences $\vect{x}$ and $\vect{y}$ of self-adjoint generators such that each $q_l$ contains at least one term from $\vect{y}$ (that is, $q_l$ is not just a NC polynomial in $\vect{x}$), and such that
\[
B \cong \CAlgebraBig{
\vect{x}=(x_0,x_1,\ldots) \\
\vect{y}=(y_0,y_1,\ldots)
}{
(R):\ -c_k\leq x_k\leq c_k, \ p_k(\vect{x})=0, \quad k\in\NN \\
(S_{\text{norm}}):\ -d_j\leq y_j\leq d_j, \quad j\in\NN \\
(S_{\text{pol}}):\ q_l(\vect{x},\vect{y})=0, \quad l\in\NN
}
\]
via an isomorphism that identifies the generator $x_k$ with the element $\bar{x}_k\in B$.
\end{lma}
\begin{proof}
We adapt the proof of \cite[Lemma~7.3]{LorShu12NCSemialgLifting}.
Set $\FF:=\QQ+i\QQ$ (a countable, dense subfield of $\CC$).
Choose a countable, dense $\FF$-$\ast$-subalgebra $B_0$ of $B$ that contains $\bar{\vect{x}}$.
Let $\bar{\vect{y}}=(\bar{y}_0,\bar{y}_1,\ldots)$ be an enumeration of the self-adjoint elements in $B_0$.
Note that each $\bar{x}_k$ appears in the sequence $\bar{\vect{y}}$, which allows us to fix a map $\alpha\colon\NN\to\NN$ such that $\bar{x}_k=\bar{y}_{\alpha(k)}$ for each $k$.
For each $j\in\NN$, set $d_j:=\|\bar{y}_j\|$.

As shown in the proof of \cite[Lemma~7.3]{LorShu12NCSemialgLifting}, there exists a countable collection $Q$ of NC polynomials such that the universal \ca{}
\[
U_1:=\CAlgebraBig{
	\vect{y}
}{
	(S_{\text{norm}}):\ -d_j\leq y_j\leq d_j, \quad j\in\NN \\
	(S_{\text{pol}}):\ q(\vect{y})=0, \quad q\in Q
}
\]
is isomorphic to $B$ via an isomorphism $\varphi\colon U_1\to B$ that maps the generator $y_k$ to the element $\bar{y}_k\in B$.

Let $(q_l)_l$ be an enumeration of the polynomials in $Q$ (considered as polynomials in $\vect{x}$ and $\vect{y}$ where no variable from $\vect{x}$ occurs) and the polynomials $x_k-y_{\alpha(k)}$, for $k\in\NN$.
Note that each $q_l$ contains at least on term from $\vect{y}$.
Set
\[
U_2:=
\CAlgebraBig{
	\vect{x}, \vect{y}
}{
(R):\ -c_k\leq x_k\leq c_k, \ p_k(\vect{x})=0, \quad k\in\NN \\
(S_{\text{norm}}):\ -d_j\leq y_j\leq d_j, \quad j\in\NN \\
(S_{\text{pol}}):\ q_l(\vect{x},\vect{y})=0, \quad l\in\NN
}.
\]
The elements $\bar{\vect{x}}$ and $\bar{\vect{y}}$ in $B$ satisfy the relations $(R)$, $(S_{\text{norm}})$ and $(S_{\text{pol}})$.
Hence, there is a (unique) \stHom{} $\psi\colon U_2\to B$ that sends $x_k$ to $\bar{x}_k$, and $y_j$ to $\bar{y}_j$.
Further, there is a \stHom{} $\beta\colon U_1\to U_2$ satisfying $\beta(y_j)=y_j$ for each $j$.
Set $\alpha:=\varphi^{-1}\circ\psi\colon U_2\to U_1$.

Note that $(\beta\circ\alpha)(y_j)=\beta(\varphi^{-1}(\bar{y}_j))=\beta(y_j)=y_j$ for each $j$, which implies that $\beta\circ\alpha$ is the identity on $U_1$.
The image of $\beta$ contains $\vect{y}$.
Using that $(S_{\text{pol}})$ contains the polynomial $x_k-y_{\alpha(k)}$ for each $k$, we deduce that the image of $\beta$ also contains $\vect{x}$, whence $\beta$ is surjective.
Together with $\beta\circ\alpha=\id_{U_1}$ this implies that $\alpha$ and $\beta$ are isomorphisms.
Hence, $\psi=\varphi\circ\beta^{-1}\colon U_2\to B$ is an isomorphism and it satisfies $\psi(x_k)=\bar{x}_k$ for each $k$, as desired.
\end{proof}

\begin{pgr}
\label{pgr:mappinCyl}
Given \ca{s} $A$ and $B$ and a morphism $\varphi\colon A\to B$, recall that the \emph{mapping cylinder} of $\varphi$, denoted $Z_\varphi$, is the pullback of $A$ and $C([0,1],B)$ along the maps $\varphi$ and $\ev_0$ (evaluation at $0$).
We have
\[
Z_\varphi = \big\{ (a,f) \in A \oplus C([0,1],B) : \varphi(a)=f(0) \big\}.
\]
Further, $Z_\varphi$ fits into the following commutative diagram:
\[
\xymatrix{
Z_\varphi \ar[r] \ar[d] & C([0,1],B) \ar[d]^{\ev_0} \\
A \ar[r]_{\varphi} & B\ .
}
\]
\end{pgr}

\begin{lma}
\label{prp:presentation_mapping_cylinder}
Let $A$ and $B$ be separable \ca{s}, and let $\varphi\colon A\to B$ be a morphism.
Assume that $A$ is given via self-adjoint generators and relations as
\[
A =\CAlgebra{\vect{x}=(x_0,x_1,\ldots)}{(R):\ -c_k\leq x_k\leq c_k,\ q_k(\vect{x})=0, \quad k\in\NN }.
\]
Set $\bar{x}_k:=\varphi(x_k)$ for each $k\in\NN$.

Apply \autoref{prp:presentation-extending} to obtain a sequence $(d_j)_j$ of positive real numbers, a sequence $(q_l)_l$ of NC polynomials in two sequences $\vect{x}$ and $\vect{y}$ of self-adjoint generators such that each $q_l$ contains at least one term from $\vect{y}$ (that is, $q_l$ is not just a NC polynomial in $\vect{x}$), and such that
\[
B \cong \CAlgebraBig{
\vect{x}=(x_0,x_1,\ldots) \\
\vect{y}=(y_0,y_1,\ldots)
}{
(R):\ -c_k\leq x_k\leq c_k, \ q_k(\vect{x})=0, \quad k\in\NN \\
(S_{\text{norm}}):\ -d_j\leq y_j\leq d_j, \quad j\in\NN \\
(S_{\text{pol}}):\ q_l(\vect{x},\vect{y})=0, \quad l\in\NN
}.
\]
via an isomorphism that identifies the generator $x_k$ with the element $\bar{x}_k\in B$.

Let $\varphi^+\colon A\to B^+$ denote the composition of $\varphi$ with the inclusion $B\subseteq B^+$.
Then the mapping cylinder $Z_{\varphi^+}$ has a presentation as
\begin{align*}
Z_{\varphi^+}
\cong\CAlgebraBig
{
\vect{x}, \vect{y}, h
}{
(R):\ -c_k\leq x_k\leq c_k,\ q_k(\vect{x})=0, \quad k\in\NN \\
(\widetilde{S}_{\text{norm}}):\ -d_jh\leq y_j\leq d_jh, \quad j\in\NN \\
(\widetilde{S}_{\text{pol}}):\ \widetilde{q}_l(\vect{x},\vect{y},h)=0, \quad l\in\NN \\
(C):\ 0\leq h\leq 1,\ hx_k=x_kh,\ hy_k=y_kh, \quad k\in\NN
},
\end{align*}
where for each $l\in\NN$ the polynomial $\widetilde{q}_l(\vect{x},\vect{y},h)$ is obtained from $q_l$ by `homogenizing' on the left with $h$ in the $y$-variables, that is, if $q_l(\vect{x},\vect{y})=\sum_{d=0}^{N_l}q_{l,d}(\vect{x},\vect{y})$ with $N_l\geq 1$ and where each polynomial $q_{l,d}$ is $d$-homogeneous in $\vect{y}$, then $\widetilde{q}_l(\vect{x},\vect{y},h):=\sum_{d=0}^{N_l}h^{N_l-d}q_{l,d}(\vect{x},\vect{y})$.
\end{lma}
\begin{proof}
The proof goes along the lines of \cite[Lemma~7.1]{LorShu12NCSemialgLifting}.
Let $U$ be the universal \ca{} that we want to show is isomorphic to $Z_{\varphi^+}$.
Recall that
\[
Z_{\varphi^+} = \big\{ (a,f) \in A\oplus C([0,1],B^+) : \varphi(a)=f(0) \big\}.
\]

To clarify the notation, we write $a\oplus f$ for an element $(a,f)\in Z_{\varphi^+}$.
For $k\in\NN$, we let $\varphi(x_k)\in C([0,1],B^+)$ denote the constant function with value $\varphi(x_k)$, and we let $t y_k \in C([0,1],B^+)$ denote the function $[0,1]\to B^+$ given by $t\mapsto t y_k$.
Define elements of $Z_{\varphi^+}$ as
\[
\tilde{x}_k := x_k \oplus \varphi(x_k), \quad
\tilde{y}_k := 0 \oplus t y_k, \quad
\tilde{h} := 0 \oplus t 1_{B^+}.
\]
Define a map $\omega\colon U\to Z_{\varphi^+}$ on the generators of $U$ by
\begin{align*}
x_k \mapsto \tilde{x}_k, \quad
y_j \mapsto \tilde{y}_j, \quad
h \mapsto \tilde{h}.
\end{align*}
To show that this assignment defines a morphism, we need to verify that $\tilde{\vect{x}}:=(\tilde{x}_0,\tilde{x}_1,\ldots)$, $\tilde{\vect{y}}:=(\tilde{y}_0,\tilde{y}_1,\ldots)$, and $\tilde{h}$ satisfy the relations defining $U$.

The relations $(R)$ and $(C)$ are clearly satisfied.
To verify $(\widetilde{S}_{\text{norm}})$, let $j\in\NN$.
Since $-d_j\leq y_j\leq d_j$ holds in $B$, we deduce that $t(-d_j)\leq ty_j\leq td_j$ for every $t\in[0,1]$, and hence
\begin{align*}
-d_j\tilde{h}
= 0\oplus t(-d_j)
\leq 0\oplus ty_j
=\tilde{y}_j
\leq 0\oplus td_j
=d_j\tilde{h}.
\end{align*}

To verify $(\widetilde{S}_{\text{pol}})$, let $l\in\NN$.
We decompose $q_l$ as $q_l(\vect{x},\vect{y})=\sum_{d=0}^{N_l}q_{l,d}(\vect{x},\vect{y})$ with $N_l\geq 1$ and where each polynomial $q_{l,d}$ is $d$-homogeneous in $\vect{y}$.
Then
\[
\widetilde{q}_l(\vect{x},\vect{y},h)
= \sum_{d=0}^{N_l}h^{N_l-d}q_{l,d}(\vect{x},\vect{y}).
\]
(For example, for $q=y_0+y_1x_1-x_0$ we obtain $\widetilde{q}=y_0+y_1x_1-hx_0$, and for $q=y_1^2+y_0-x_1^3$ we obtain $\widetilde{q}=y_1^2+hy_0-h^2x_1^3$.)
Using at the fourth step that $q_{l,d}(\varphi(\vect{x}),t\vect{y})=t^d q_{l,d}(\varphi(\vect{x}),\vect{y})$, and using at the last step that $q_l(\varphi(\vect{x}),\vect{y})=q_l(\bar{\vect{x}},\vect{y})=0$, we deduce that
\begin{align*}
\widetilde{q}_l(\tilde{\vect{x}},\tilde{\vect{y}},\tilde{h})
&= \sum_{d=0}^{N_l}\tilde{h}^{N_l-d} q_{l,d}(\tilde{\vect{x}},\tilde{\vect{y}}) \\
&= \sum_{d=0}^{N_l}(0\oplus t)^{N_l-d}q_{l,d}(\vect{x}\oplus\varphi(\vect{x}),0\oplus t\vect{y}) \\
&= \left( \sum_{d=0}^{N_l} 0^{N_l-d} q_{l,d}(\vect{x},0) \right) \oplus \left( \sum_{d=0}^{N_l}t^{N_l-d} q_{l,d}(\varphi(\vect{x}),t\vect{y}) \right) \\
&= q_{l,N_l}(\vect{x},\vect{0}) \oplus \left( \sum_{d=0}^{N_l}t^{N_l} q_{l,d}(\varphi(\vect{x}),\vect{y}) \right) \\
&=0\oplus t^{N_l} q_l(\varphi(\vect{x}),\vect{y}) \\
&=0\oplus 0,
\end{align*}
as desired.
Thus, $\omega\colon U\to Z_\varphi$ is a well-defined morphism.
We proceed to show that $\omega$ is bijective.

To show that $\omega$ is surjective, note that every $z=a\oplus f\in Z_{\varphi^+}$ can be written as
\[
z = (a\oplus\varphi(a)) + (0\oplus  (f-\varphi(a))).
\]
We have $f-\varphi(a)\in C_0((0,1],B^+)$.
Thus, it is enough to show that the image of $\omega$ contains $a\oplus\varphi(a)$, for $a\in A$, and $0\oplus f$, for $f\in C_0((0,1],B^+)$.

Let $a\in A$.
Since $\vect{x}$ generates $A$, we can choose a sequence of NC polynomials $(r_n)_n$ such that $\lim_n \|a-r_n(\vect{x})\|=0$.
Then
\[
\lim_n \|(a\oplus\varphi(a))-r_n(\tilde{\vect{x}})\|=0.
\]
Since each $r_n(\tilde{\vect{x}})$ belongs to the image of $\omega$, we obtain that $a\oplus\varphi(a)$ belongs to the image of $\omega$, as desired.

On the other hand, as in the proof of \cite[Lemma~7.1]{LorShu12NCSemialgLifting}, to show that $0\oplus C_0((0,1],B^+)$ belongs to the image of $\omega$, it is enough to verify that $0\oplus f$ is in the image of $\omega$ for $f$ given by $f(t)=t^s(ty_{j_1})(ty_{j_2})\ldots(ty_{j_n})$, for every $s,n\in\NN$ and $j_1,\ldots,j_n\in\NN$.
This follows since
\[
0\oplus f
= (\tilde{h})^s \tilde{y}_{j_1} \tilde{y}_{j_2} \ldots \tilde{y}_{j_n}.
\]

To show that $\omega$ is injective, let $z\in U$ with $z\neq 0$.
Choose an irreducible representation $\sigma\colon U\to\Bdd(K)$ with $\sigma(z)\neq 0$.
Set
\[
\vect{X}:=(\sigma(x_0),\sigma(x_1),\ldots),\andSep
\vect{Y}:=(\sigma(y_0),\sigma(y_1),\ldots),\andSep
H:=\sigma(h).
\]

The relation $(C)$ tells us that $H$ is a positive contraction that commutes with all operators in the image of $\sigma$.
Since $\sigma$ is irreducible, $H$ is a scalar multiple of the identity operator.
Hence, $H=\lambda 1$ for some $\lambda\in[0,1]$.
We distinguish the two cases $\lambda=0$ and $\lambda>0$.

Let $\pi\colon Z_{\varphi^+}\to A$ be given by $\pi(a\oplus f):=a$.
Then $\pi$ is a surjective morphism.
The kernel of $\pi$ is naturally identified with $C_0((0,1],B^+)$.
We have the the following short exact sequence
\begin{align*}
0\to C_0((0,1],B^+)\xrightarrow{\iota} Z_{\varphi^+}\xrightarrow{\pi}A\to 0.
\end{align*}

Case~1.
Assume that $\lambda=0$.
Then $H=0$, and the relations $(\widetilde{S}_{\text{norm}})$ imply that $\vect{Y}=0$.
Then $\vect{X}$ is a representation of $\langle\vect{x}|(R)\rangle$.
Let $\tau\colon A=\CAlgebra{\vect{x}}{(R)}\to\Bdd(K)$ be the induced morphism.
One checks that $\sigma$ agrees with $\tau\circ\pi\circ\omega$ on each generator of $U$, which implies that $\sigma=\tau\circ\pi\circ\omega$.
It follows that $\omega(z)\neq 0$, as desired.

Case~2.
Assume that $\lambda>0$.
Then $H=\lambda$.
Let us verify that $(\vect{X},\lambda^{-1}\vect{Y})$ is a representation of $\langle \vect{x},\vect{y} | (R),(S_{\text{norm}}),(S_{\text{pol}})\rangle$.
It is clear that $\vect{X}$ satisfy $(R)$.
Since $\vect{Y}$ satisfy $(\widetilde{S}_{\text{norm}})$ and $H=\lambda$, we have
\[
-d_j\lambda = -d_jH \leq Y_j\leq d_jH = d_j\lambda,
\]
and therefore $-d_j\leq\lambda^{-1}Y_j\leq d_j$, for every $j\in\NN$.
Thus, $\lambda^{-1}\vect{Y}$ satisfy $(S_{\text{norm}})$.

To verify $(S_{\text{pol}})$ for $(\vect{X},\lambda^{-1}\vect{Y})$, let $l\in\NN$.
We decompose $q_l$ according to the degree of homogeneity in $\vect{y}$ as above and compute
\begin{align*}
q_l(\vect{X},\lambda^{-1}\vect{Y})
&=\sum_{d=0}^{N_l}q_{l,d}(\vect{X},\lambda^{-1}\vect{Y}) \\
&=\sum_{d=0}^{N_l}\lambda^{-d}q_{l,d}(\vect{X},\vect{Y}) \\
&=\lambda^{-N_l}\sum_{d=0}^{N_l}\lambda^{N_l-d}q_{l,d}(\vect{X},\vect{Y}) \\
&=\lambda^{-N_l}\widetilde{q}_l(\vect{X},\vect{Y},H) \\
&= 0.
\end{align*}

Let $\tau\colon B\cong \CAlgebra{\vect{x},\vect{y}}{(R), (S_{\text{norm}}), (S_{\text{pol}})}\to\Bdd(K)$ be the induced morphism, and let $\tau^+\colon B^+\to\Bdd(K)$ be the unique extension to a unital morphism.
Let $\ev_\lambda\colon Z_{\varphi^+}\to B^+$ be given by $\ev_\lambda(a\oplus f):=f(\lambda)$.
We claim that $\sigma=\tau^+\circ\ev_\lambda\circ\omega$.
It is enough to verify this on the generators of $U$.
We have
\[
(\tau^+\circ\ev_\lambda\circ\omega)(h)
= (\tau^+\circ\ev_\lambda)(0\oplus t)
= \tau^+(\lambda)
= \lambda = H = \sigma(h).
\]
Further,
\[
(\tau^+\circ\ev_\lambda\circ\omega)(y_j)
= (\tau^+\circ\ev_\lambda)(0\oplus ty_j)
= \tau^+(\lambda y_j)
= \lambda \tau(y_j)
= \lambda \lambda^{-1} Y_j
= \sigma(y_j),
\]
for each $j\in\NN$.
Moreover,
\[
(\tau^+\circ\ev_\lambda\circ\omega)(x_k)
= (\tau^+\circ\ev_\lambda)(x_k\oplus \varphi(x_k))
= \tau^+(\varphi(x_k))
= \tau(\bar{x}_k)
= X_k = \sigma(x_k),
\]
for each $k\in\NN$.
Thus, $\sigma=\tau^+\circ\ev_\lambda\circ\omega$.
It follows that $\omega(z)\neq 0$, as desired.
\end{proof}

In the next result, we use $[x,y]$ to denote the commutator $[x,y]:=xy-yx$.

\begin{lma}
\label{prp:softened_mapping_cylinder_SP}
We retain the notation from \autoref{prp:presentation_mapping_cylinder}.
Given $n\in\NN$, set
\begin{align*}
Z_{\varphi^+}^{(n)}
:=\CAlgebraBig
{
\vect{x} \\ \vect{y} \\ h
}{
(R):\ -c_k\leq x_k\leq c_k,\ q_k(\vect{x})=0, \quad k\in\NN \\
0\leq h\leq 1 \\
(\widetilde{S}_{\text{norm}}):\ -d_jh\leq y_j\leq d_jh, \quad j\in\NN \\
(\widetilde{S}_{\text{pol}}^n):\ \|\widetilde{q}_l(\vect{x},\vect{y},h)\|\leq 1/n, \quad l=0,1,\ldots,n \\
(C^n):\ \|[h,x_k]\|, \|[h,y_k]\|\leq 1/n, \quad k=0,1,\ldots,n
}.
\end{align*}
Assume that $A=\CAlgebra{\vect{x}}{(R)}$ is semiprojective.
Then $Z_{\varphi^+}^{(n)}$ is semiprojective.
\end{lma}
\begin{proof}
The proof goes along the lines of \cite[Lemma~7.2]{LorShu12NCSemialgLifting}.
Let $C$ be a \ca{} with an increasing sequence $J_0\lhd J_1\lhd\ldots \lhd C$ of ideals and set $J:=\overline{\bigcup_mJ_m}$.
Let $\vect{x}=(x_0,x_1,\ldots)$, $\vect{y}=(y_0,y_1,\ldots)$, and $h$ be (sequences) of self-adjoint elements in $C/J$ satisfying the relations defining $Z_{\varphi^+}^{(n)}$.
We need to find $m\in\NN$ and lifts $\tilde{\vect{x}}$, $\tilde{\vect{y}}$, and $\tilde{h}$ in $C/J_m$ that satisfy the same relations.

Since $\CAlgebra{\vect{x}}{(R)}$ is semiprojective, there exist $m$ and a sequence $\tilde{\vect{x}}=(\tilde{x}_0,\tilde{x}_1,\ldots)$ of self-adjoint elements in $C/J_m$ that satisfy $(R)$ and that lift $\vect{x}$.
We may also find a lift $\tilde{h}\in C/J_m$ of $h$ such that $0\leq\tilde{h}\leq 1$.
Applying Davidson's order lifting theorem, \cite[Corollary~2.2]{Dav91LifingPosElts}, see also  \cite[Corollary~8.2.3, p.63]{Lor97LiftingSolutions}, we find a lift $\tilde{\vect{y}}$ of $\vect{y}$ in $C/J_m$ that satisfies $(\widetilde{S}_{\text{norm}})$, that is, such that $-d_j\tilde{h}\leq\tilde{y}_j\leq d_j\tilde{h}$ for every $j\in\NN$.

Note that the finitely many polynomials defining $(\widetilde{S}_{\text{pol}}^n)$ and $(C^n)$ involve only finitely many variables and are homogeneous (of degree at least one) in the variables $\vect{y}, h$ (but not necessarily in $\vect{x}$).
Let $\pi\colon C/J_m\to C/J$ denote the quotient morphism.
We have
\[
\left\| \widetilde{q}_l \big( \pi(\tilde{\vect{x}}),\pi(\tilde{\vect{y}}),\pi(\tilde{h}) \big) \right\|
= \left\| \widetilde{q}_l \big( \vect{x},\vect{y},h \big) \right\| \leq 1/n,
\]
for $l=0,\ldots,n$, and
\[
\left\| \big[ \pi(\tilde{h}),\pi(\tilde{x}_k) \big] \right\|
= \left\| \big[ h,x_k \big] \right\|\leq 1/n,\andSep
\left\| \big[ \pi(\tilde{h}),\pi(\tilde{y}_k)\big] \right\|
= \left\| \big[ h,y_k \big] \right\|\leq 1/n,
\]
for $k=0,\ldots,n$.
By \cite[Theorem~3.2]{LorShu12NCSemialgLifting} there exists $e\in J_m+1$ with $0\leq e\leq 1$ such that $\tilde{h}':=e\tilde{h}e$ and the sequence $\tilde{\vect{y}}' := e\tilde{\vect{y}}e = (e\tilde{y_0}e,e\tilde{y_1}e,\ldots)$ lift $h$ and $\vect{y}$, and such that $\tilde{h}'$, $\tilde{\vect{y}}'$ and $\tilde{\vect{x}}$ satisfy $(\widetilde{S}_{\text{pol}}^n)$ and $(C^n)$.
Note that $\tilde{h}'$ and $\tilde{\vect{y}}'$ also satisfy $(\widetilde{S}_{\text{norm}})$ as
\[
-d_j \tilde{h}'
= -d_j e\tilde{h}e
= e(-d_j\tilde{h})e
\leq e\tilde{y}_je = \tilde{y}_j' 
\leq e(-d_j\tilde{h})e
=d_j \tilde{h}',
\]
for every $j\in\NN$.
This shows the semiprojectivity of $Z_{\varphi^+}^{(n)}$.
\end{proof}

The following is the main technical result of this paper.
It shows that certain mapping cylinders are isomorphic to inductive limits of semiprojective \ca{s}.

\begin{thm}
\label{prp:mapping_cylinder_limit_SP}
Let $A$ and $B$ be separable \ca{s}, and let $\varphi\colon A\to B$ be a morphism.
If $A$ is semiprojective, then the mapping cylinder $Z_\varphi$ has a strong shape system.
\end{thm}
\begin{proof}
The proof goes along the lines of \cite[Theorem~7.4]{LorShu12NCSemialgLifting}.
Let $\varphi^+\colon A\to B^+$ denote the composition of $\varphi$ with $B\subseteq B^+$, as in \autoref{prp:presentation_mapping_cylinder}.
Recall that
\[
Z_{\varphi^+} = \big\{ (a,f) \in A\oplus C([0,1],B^+) : \varphi(a)=f(0) \big\}.
\]
Let $\varrho\colon B^+\to\CC$ be the natural quotient morphism, such that $B=\ker(\varrho)$.
Given $(a,f)\in Z_{\varphi^+}$ we have $f(0)\in B$ and therefore $\varrho(f(0))=0$.
We may therefore define $\pi\colon Z_{\varphi^+}\to C_0((0,1])$ by $\pi((a,f)):=\varrho\circ f$, for $a\in A$ and $f\in C([0,1],B^+)$.

The natural inclusion $C([0,1],B)\to C([0,1],B^+)$ induces an injective morphism $\iota\colon Z_\varphi\to Z_{\varphi^+}$, which we use to identify $Z_\varphi$ with a sub-\ca{} of $Z_{\varphi^+}$.
The kernel of $\pi$ is $Z_\varphi$.
Hence, we have a short exact sequence:
\begin{align*}
0\to Z_\varphi \xrightarrow{\iota}  Z_{\varphi^+} \xrightarrow{\pi} C_0((0,1])\to 0.
\end{align*}

For $n\in\NN$, let $Z_{\varphi^+}^{(n)}$ be defined as in \autoref{prp:softened_mapping_cylinder_SP}.
By construction, we have natural surjective morphisms $\gamma_{n+1,n}\colon Z_{\varphi^+}^{(n)}\to Z_{\varphi^+}^{(n+1)}$ such that the resulting inductive limit is isomorphic to $Z_{\varphi^+}$.
Let $\gamma_{\infty,n}\colon Z_{\varphi^+}^{(n)}\to Z_{\varphi^+}$ be the surjective morphism to the inductive limit.
For each $n\in\NN$, set
\[
D_n := \ker(\pi\circ\gamma_{\infty,n}).
\]
The morphism $\gamma_{n+1,n}$ induces a morphism $\psi_{n+1,n}\colon D_n\to D_{n+1}$.
Then $Z_\varphi\cong\varinjlim_n D_n$.
The situation is shown in the following commutative diagram, whose rows are short exact sequences:
\[
\xymatrix{
0 \ar[r] & Z_\varphi \ar[r]^{\iota} & Z_{\varphi^+} \ar[r]^-{\pi} & C_0((0,1]) \ar[r] & 0 \\
0 \ar[r] & D_{n+1} \ar[r] \ar@{->>}[u]^{\psi_{\infty,n}} & Z_{\varphi^+}^{(n+1)} \ar[r] \ar@{->>}[u]^{\gamma_{\infty,n}} & C_0((0,1]) \ar[r] \ar@{=}[u] & 0 \\
0 \ar[r] & D_n \ar[r] \ar@{->>}[u]^{\psi_{n+1,n}} & Z_{\varphi^+}^{(n)} \ar[r] \ar@{->>}[u]^{\gamma_{n+1,n}} & C_0((0,1]) \ar[r] \ar@{=}[u] & 0 \ .
}
\]

By \autoref{prp:softened_mapping_cylinder_SP}, each $Z_{\varphi^+}^{(n)}$ is semiprojective.
An ideal $J$ in a semiprojective \ca{} $E$ is semiprojective if the quotient $E/J$ is projective;
see \cite[Corollary~3.1.3]{End16CharSemiprojSubhom}, which generalizes \cite[Theorem~5.3]{LorPed98ProjAF-telescope}.
Since $C_0((0,1])$ is projective, it follows that each $D_n$ is semiprojective.
Thus, $Z_\varphi$ is isomorphic to an inductive limit of semiprojective \ca{s}, as desired.
\end{proof}

\section{\texorpdfstring{$C^*$-algebras}{C*-algebras} with strong shape system}
\label{sec:sss}

In this section, we study closure properties of the class of separable \ca{s} that have a strong shape system.
Recall that a separable \ca{} is said to have a strong shape system if it is isomorphic to a sequential inductive limit of separable, semiprojective \ca{s};
see \cite[Definition~4.1]{Bla85ShapeThy}.
The main result of this section (and the whole paper) is \autoref{prp:sss_shDomination}:
The class separable \ca{s} that have a strong shape system is closed under shape domination.
Hence, if two separable \ca{s} are shape equivalent (in particular, if they are homotopy equivalent), then one has a strong shape system if and only if the other does;
see \autoref{prp:sss_shEquivalence}.

This provides many new examples of \ca{s} with a strong shape system.
For example, $C([0,1]^n,A)$ has a strong shape system for every UCT-Kirchberg algebra~$A$;
see \autoref{exa:ShKirchberg}.

\begin{prp}
\label{prp:sssUnitization}
A separable \ca{} $A$ has a strong shape system if and only if its minimal unitization $\widetilde{A}$ does.
\end{prp}
\begin{proof}
If $A$ is unital, then there is nothing to show.
So assume that $A$ is non-unital.
Note that a \ca{} $B$ is semiprojective if and only if $\widetilde{B}$ is;
see \cite[Theorem~14.1.7, p.108]{Lor97LiftingSolutions}.

To show the forward implication, assume that $A=\varinjlim_n A_n$, with each $A_n$ a semiprojective \ca{}.
Then $\widetilde{A}\cong\varinjlim_n \widetilde{A}_n$, and each $\widetilde{A}_n$ is semiprojective.

To show the backward implication, assume that $\widetilde{A}=\varinjlim_n B_n$, with each $B_n$ a semiprojective \ca{}.
Let $\gamma_{n+1,n}\colon B_n\to B_{n+1}$ be the connecting morphisms.
By \cite[Proposition~4.8]{Thi11arX:IndLimPj}, we may assume that each $\gamma_{n+1,n}$ is surjective.
Let $\pi\colon\widetilde{A}\to\CC$ be the quotient map such that $A=\ker(\pi)$.
For each $n\in\NN$, set $A_n:=\ker(\pi\circ\gamma_{\infty,n})$.
The morphism $\gamma_{n+1,n}$ induces a morphism $\psi_{n+1,n}\colon A_n\to A_{n+1}$.
Then $A\cong\varinjlim_n A_n$.
The situation is shown in the following commutative diagram, whose rows are short exact sequences:
\[
\xymatrix{
0 \ar[r] & A \ar[r]^{\iota} & \widetilde{A} \ar[r]^{\pi} & \CC \ar[r] & 0 \\
0 \ar[r] & A_{n+1} \ar[r] \ar[u]^{\psi_{\infty,n+1}} & B_{n+1} \ar[r] \ar[u]^{\gamma_{\infty,n+1}} & \CC \ar[r] \ar@{=}[u] & 0 \\
0 \ar[r] & A_n \ar[r] \ar[u]^{\psi_{n+1,n}} & B_n \ar[r] \ar[u]^{\gamma_{n+1,n}} & \CC \ar[r] \ar@{=}[u] & 0 \\
}.
\]
For each $n$, the algebra $A_n$ is an ideal of finite codimension in the semiprojective \ca{} $B_n$.
It follows from \cite[Theorem~3.2]{End17SemiprojFiniteCodim} that $A_n$ is semiprojective.
Hence, $A$ has a strong shape system, as desired.
\end{proof}

By \cite[Corollary~4.5]{Thi11arX:IndLimPj}, every separable, contractible \ca{} is an inductive limit of projective \ca{s}.
Combining this result with \autoref{prp:sssUnitization}, we obtain:

\begin{cor}
Let $A$ be a contractible \ca{s}.
Then $\widetilde{A}$ has a strong shape system.
\end{cor}

\begin{lma}
\label{prp:factorization}
Let $A$, $B$ and $C$ be \ca{s}, and let $\gamma\colon A\to C$, $\alpha\colon A\to B$ and $\beta\colon B\to C$ be morphisms satisfying $\gamma\simeq\beta\circ\alpha$.
Then there exist morphisms $\varphi\colon A\to Z_\beta$ and $\omega\colon Z_\beta\to C$ such that $\gamma=\omega\circ\varphi$.
\end{lma}
\begin{proof}
The homotopy $\gamma\simeq\beta\circ\alpha$ is given by a morphism $\psi\colon A\to C([0,1],C)$ satisfying
\[
\beta\circ\alpha=\ev_0\circ\psi,\andSep \gamma=\ev_1\circ\psi.
\]
The mapping cylinder $Z_\beta$ is the pullback of $B$ and $C([0,1],C)$ along $\beta$ and $\ev_0$.
Let $\delta\colon Z_\beta\to C([0,1],C)$ be the natural morphism from the pullback.
By the universal property of pullbacks, the morphisms $\alpha$ and $\psi$ induce a morphism $\varphi\colon A\to Z_\beta$ such that $\delta\circ\varphi=\psi$.
The situation is shown in the following commutative diagram:
\[
\xymatrix{
A \ar[rrr]^{\gamma} \ar@/_0.5pc/[ddr]_{\alpha} \ar@{-->}[dr]^{\varphi} \ar@/^0.5pc/[drr]^{\psi} & & & C \\
& Z_\beta \ar[r]^-{\delta} \ar[d] & C([0,1],C) \ar[ur]_{\ev_1} \ar[d]^{\ev_0} \\
& B \ar[r]_{\beta} & C\ .
}
\]
Set $\omega:=\ev_1\circ\delta$.
Then
\[
\omega\circ\varphi
= \ev_1\circ\delta\circ\varphi
= \ev_1\circ\psi
= \gamma,
\]
as desired.
\end{proof}

\begin{thm}
\label{prp:sss_shDomination}
Let $A$ and $B$ be separable \ca{s} with $A\precsim_{\Sh}B$.
Then, if $B$ has a strong shape system, so does $A$.
\end{thm}
\begin{proof}
Let $(A_n,\gamma_{n+1,n})$ be a shape system for $A$, and let $(B_n,\theta_{n+1,n})$ be a strong shape system for $B$.
Using that $A\precsim_{\Sh}B$, it follows from \cite[Theorem~4.8]{Bla85ShapeThy} that $(A_n,\gamma_{n+1,n})$ is shape dominated by $(B_n,\theta_{n+1,n})$.
Thus, after reindexing, we may assume that there are morphisms $\alpha_n\colon A_n\to B_n$ and $\beta_n\colon B_n\to A_{n+1}$ such that $\gamma_{n+1,n}\simeq\beta_n\circ\alpha_n$ for all $n\in\NN$.
For each $n\in\NN$, applying \autoref{prp:factorization}, we obtain morphisms $\varphi_n\colon A_n\to Z_{\beta_n}$ and $\omega_n\colon Z_{\beta_n}\to A_{n+1}$ such that $\gamma_{n+1,n}=\omega_n\circ\varphi_n$.
Set $\psi_{n+1,n}:=\varphi_{n+1}\circ\omega_n\colon Z_{\beta_n}\to Z_{\beta_{n+1}}$.
Then the inductive systems $(A_n,\gamma_{n+1,n})$ and $(Z_{\beta_n},\psi_{n+1,n})$ are intertwined, as shown in the following commutative diagram:
\[
\xymatrix{
A_n \ar[dr]_{\varphi_n} \ar[rr]^{\gamma_{n+1,n}}
& & A_{n+1} \ar[dr]_{\varphi_{n+1}} \ar[rr]^{\gamma_{n+2,n+1}}
& & A_{n+2} \ar[dr]_{\varphi_{n+2}} & \ldots \\
& Z_{\beta_n} \ar[ur]^{\omega_n} \ar[rr]_{\psi_{n+1,n}}
& & Z_{\beta_{n+1}} \ar[ur]^{\omega_{n+1}} \ar[rr]_{\psi_{n+2,n+1}}
& & Z_{\beta_{n+2}}\ .
}
\]

It follows that $A\cong\varinjlim_n Z_{\beta_n}$.
By \autoref{prp:mapping_cylinder_limit_SP}, each $Z_{\beta_n}$ is an inductive limit of semiprojective \ca{s}.
Applying \cite[Theorem~3.12]{Thi11arX:IndLimPj}, it follows that $A$ is an inductive limit of semiprojective \ca{s}, as desired.
\end{proof}

\begin{cor}
\label{prp:sss_shEquivalence}
Let $A$ and $B$ be shape equivalent, separable \ca{s}.
Then $A$ has a strong shape system if and only if $B$ does.
\end{cor}

\begin{exa}
\label{exa:sss_S2K}
Let $A$ be a \ca.
We set $\Sigma:=C_0(\RR)$.
Further, $\Sigma A:=C_0(\RR)\otimes A$ denotes the suspension of $A$, and similarly $\Sigma^2A:=C_0(\RR^2)\otimes A$ is the second suspension of $A$.

The \ca{} $qA$ was introduced by Cuntz in his study of $KK$-theory.
It is defined as the kernel of the morphism $A\ast A\to A$ obtained from the universal property of the free product $A\ast A$ applied to the identity morphism on both factors.
In particular, we have a short exact sequence
\[
0 \to qA \to A\ast A \to A \to 0.
\]

In \cite{Shu10AsymEquiv-qAK-S2AK}, Shulman showed that the \ca{s} $qA\otimes\KK$ and $\Sigma^2 A\otimes\KK$ are asymptotically equivalent.
Dadarlat showed in \cite{Dad94ShapeThyAsymMor} that the notion of asymptotic equivalence and shape equivalence agree.
Thus, $qA\otimes\KK$ and $\Sigma^2 A\otimes\KK$ are shape equivalent.
Hence, by \autoref{prp:sss_shEquivalence}, $qA\otimes\KK$ has a strong shape system if and only if $\Sigma^2 A\otimes\KK$ does.

Let us consider the case $A=\CC$.
It is known that $q\CC$ is semiprojective;
see \cite[Chapter~16]{Lor97LiftingSolutions}.
It follows that $q\CC\otimes\KK$, and consequently $\Sigma^2\CC\otimes\KK$, has a strong shape system.
Note that $\Sigma^2\CC\otimes\KK$ is isomorphic to $C_0(S^2\setminus\{\ast\})\otimes\KK$, the stabilized algebra of continuous functions on the pointed sphere.
We do not know if the stabilization is necessary, that is, if $\Sigma^2$ has a strong shape system.
By \autoref{prp:sssUnitization}, this is equivalent to the following open question.
\end{exa}

\begin{qst}
Does the \ca{} $C(S^2)$ of continuous functions on the two-sphere have a strong shape system?
\end{qst}

\begin{exa}
\label{exa:ShKirchberg}
Recall that a \emph{Kirchberg algebra} is a separable, purely infinite, simple, nuclear \ca.
To simplify, we call a Kirchberg algebra that satisfies the universal coefficient theorem (UCT) a \emph{UCT-Kirchberg algebra}.
We refer to \cite[Section~23, p.232ff]{Bla98KThy} for details on the UCT.

In \cite[Corollary~4.6]{End15arX:SemiprojKirchberg}, Enders solved a long-standing conjecture of Blackadar by showing that a UCT-Kirchberg algebra is semiprojective if and only if its $K$-groups are finitely generated.
It follows from \cite[Proposition~8.4.13]{Ror02Classification} that every UCT-Kirchberg algebra is isomorphic to an inductive limit of UCT-Kirchberg algebras with finitely generated $K$-groups.
We deduce that every UCT-Kirchberg algebra has a strong shape system.

It follows from \autoref{prp:sss_shDomination} that every separable \ca{} that is shape dominated by a UCT-Kirchberg algebra has a strong shape system.
For example, $C([0,1]^n,A)$ has a strong shape system for every UCT-Kirchberg algebra $A$ and every $n\geq 1$.
\end{exa}

The next result records permanence properties of the class of \ca{s} that are isomorphic to inductive limits of semiprojective \ca{s}.

\begin{thm}
\label{prp:sss_permanence}
The class of separable \ca{s} that have a strong shape system is closed under:
\begin{enumerate}[\quad (1)  ]
\item
countable direct sums;
\item
sequential inductive limits;
\item
approximation by sub-\ca{s} in the sense of \cite[Paragraph~3.1]{Thi11arX:IndLimPj};
\item
shape domination, in particular shape equivalence, homotopy domination and homotopy equivalence;
\item
passing to matrix algebras or stabilization.
\end{enumerate}
\end{thm}
\begin{proof}
To prove statement~(1), let $(A_n)_{n\in\NN}$ be a sequence of separable \ca{s} with strong shape systems.
For each $n\in\NN$, let $(A_{n,k},\varphi_{k+1,k}^{(n)})_{k}$ be a strong shape system for $A_n$.
Set $B:=\bigoplus_{n\in\NN} A_n$ and set $B_k:=\bigoplus_{n=0}^k A_{n,k}$ for each $k\in\NN$.
Define $\psi_{k+1,k}\colon B_k\to B_{k+1}$ by mapping the summand $A_{n,k}$ in $B_k$ to the summand $A_{n,k+1}$ in $B_{k+1}$ by the map $\varphi_{k+1,k}^{(n)}$, for each $n=0,\ldots,k$.
Then each $B_k$ is semiprojective, and $B\cong\varinjlim_k B_k$, as desired.

Statements~(2) and~(3) follow from Theorem~3.12 and Theorem~3.9 in \cite{Thi11arX:IndLimPj}, respectively.
Statement~(4) follows from \autoref{prp:sss_shDomination}.

To prove statement~(5), let $A$ be a separable \ca{} with $A\cong\varinjlim_n A_n$ for semiprojective \ca{s} $A_n$.
Given $k\in\NN$, we have $A\otimes M_k\cong\varinjlim_n A_n\otimes M_k$.
Further, we have $A\otimes\KK\cong\varinjlim_n A_n\otimes M_n$, with connecting maps $A_n\otimes M_n\to A_{n+1}\otimes M_{n+1}$ given by the amplification of the connecting map $A_n\to A_{n+1}$ to $n\times n$-matrices, followed by the upper left corner embedding $A_{n+1}\otimes M_n\to A_{n+1}\otimes M_n$.
By \cite[Theorem~14.2.2, p.110]{Lor97LiftingSolutions}, a matrix algebra of a separable, semiprojective \ca{} is again semiprojective.
This shows that $A\otimes M_k$ and $A\otimes\KK$ have strong shape systems.
\end{proof}

\section{Nuclear \texorpdfstring{$C^*$-algebras}{C*-algebras} with strong shape system}
\label{sec:nuclearSSS}

In this section, we show that every separable, stable, nuclear, homotopy symmetric \ca{} satisfying the universal coefficient theorem (UCT) in $KK$-theory and with torsion-free $K_0$-group has a strong shape system;
\autoref{prp:ssa_for_HS}.
Hence, if $A$ is a separable, nuclear \ca{} satisfying the UCT, then the stable suspension $\Sigma A\otimes\KK$ has a strong shape system if $K_1(A)$ is torsion-free, and the stable second suspension $\Sigma^2 A\otimes\KK$ has a strong shape system if $K_0(A)$ is torsion-free;
see \autoref{prp:sss_suspension_nuclear}.

The notion of `homotopy symmetry' was introduced by Dadarlat and Loring in \cite[Section~5]{DadLor94UnsuspendedEThy}, to which we refer for the definition.
For separable, nuclear \ca{s}, Dadarlat and Pennig showed in \cite[Theorem~3.1]{DadPen17DefNilpotHomSym} that homotopy symmetry is equivalent to `connectivity', which is defined via an embedding in a special \ca{} together with a lifting property:
A (separable) \ca{} $A$ is said to be \emph{connective} if there exists an injective morphism $A\to \prod_n C\Bdd(H) / \bigoplus_n C\Bdd(H)$ that admits a completely positive, contractive lift $A\to \prod_n C\Bdd(H)$, where $C\Bdd(H):=C_0((0,1],\Bdd(H))$ is the cone over $\Bdd(H)$;
see \cite[Definition~2.1]{DadPen17ConnectiveCa}.
(Note that in \cite{DadPen17DefNilpotHomSym}, connectivity is called property (QH).)

For homotopy symmetric \ca{s}, shape theory and $E$-theory are closely related:
By combining results in \cite{DadLor94UnsuspendedEThy} and \cite{Dad94ShapeThyAsymMor}, we obtain that two separable, stable, homotopy symmetric \ca{s} are shape equivalent if and only if they are $E$-equivalent;
see \autoref{prp:ShEquiv-EEquiv_for_HS}.
We refer to \cite{Bla98KThy} for details on $KK$-theory, $E$-theory and the UCT.

To prove \autoref{prp:ssa_for_HS}, we construct for every given pair $(G_0,G_1)$ of countable, abelian groups with $G_0$ torsion-free a model \ca{} $A$ that is stable, nuclear, homotopy symmetric, satisfies the UCT, has a strong shape system, and such that $K_\ast(A)\cong(G_0,G_1)$;
see \autoref{prp:realizeKThy_sssHs}.
It is not clear if such model \ca{s} exist with torsion in $K_0$;
see \autoref{qst:exist_sssHS} and \autoref{rmk:exist_sssHS}.

\begin{pgr}
We set $\Sigma:=C_0(\RR)$.
Given $n\in\NN$ with $n\geq 2$, we define
\begin{align*}
I_n  := \big\{ f\in C((0,1],M_n) : f(1)\in\CC 1_{M_n} \big\}.
\end{align*}
The \ca{} $I_n$ is called the (nonunital) \emph{dimension-drop algebra} of order $n$.
We have
\[
K_\ast(\Sigma)\cong(0,\ZZ),\andSep K_\ast(I_n)\cong(0,\ZZ_n).
\]
The \ca{s} $\Sigma$ and $I_n$ are homotopy symmetric;
see Propositions~5.3 and~6.1 in \cite{DadLor94UnsuspendedEThy}.
\end{pgr}

\begin{lma}
\label{prp:realizeKThy_sssHs}
Let $G_0$ and $G_1$ be countable, abelian groups.
Assume that $G_0$ is torsion-free.
Then there exists a stable, nuclear, homotopy symmetric \ca{} $A$ satisfying the UCT, with a strong shape system, and such that $K_\ast(A)\cong(G_0,G_1)$.
\end{lma}
\begin{proof}
We will construct stable, nuclear, homotopy symmetric \ca{s} $A_0$ and $A_1$ that satisfy the UCT, that have strong shape systems, and that satisfy $K_\ast(A_0)\cong(G_0,0)$ and $K_\ast(A_1)\cong(0,G_1)$.
Since the properties of being stable, nuclear, homotopy symmetric, satisfying the UCT, and having a strong shape system are preserved by direct sums, it will follow that $A:=A_0\oplus A_1$ has the desired properties.

To construct $A_0$, we use that every countable, torsion-free, abelian group arises as the $K_0$-group of a separable AF-algebra;
see for example \cite[Corollary~23.10.4, p.241]{Bla98KThy}.
Thus, we may choose a separable AF-algebra $B$ with $K_0(B)\cong G_0$.
Set $A_0:=\Sigma^2 B \otimes \KK$.
Then $A_0$ is clearly nuclear and satisfies $K_0(A_0)\cong (G_0,0)$.
By \cite[Lemma~5.1]{DadLor94UnsuspendedEThy}, the tensor product of a (nuclear) homotopy symmetric \ca{} with any other \ca{} is again homotopy symmetric.
Thus, since $\Sigma$ is homotopy symmetric, so is $A_0$.

Let $F_n$ be finite-dimensional \ca{s} such that $B\cong\varinjlim_n F_n$.
Then
\[
\Sigma^2 B\otimes\KK \cong \varinjlim_n (\Sigma^2 F_n \otimes \KK).
\]
Thus, by \autoref{prp:sss_permanence}(2), it is enough to verify that each $\Sigma^2 F_n \otimes \KK$ has a strong shape system.
Since $F_n$ is a direct sum of matrix algebras, applying \autoref{prp:sss_permanence}(1), it is even enough to show that $\Sigma^2 M_k \otimes \KK$ has a strong shape system.
This follows from \autoref{exa:sss_S2K} since $\Sigma^2 M_k \otimes \KK \cong \Sigma^2\CC\otimes\KK$.

To construct $A_1$, choose finitely generated, abelian groups $H_n$ and group homomorphisms $\varphi_{n+1,n}\colon H_n\to H_{n+1}$, for $n\in\NN$, such that $G_1\cong \varinjlim_n H_n$.
Every finitely generated, abelian group is a finite direct sum of cyclic groups.
Thus, for each $n$ there exist $e_n\in\NN$ and cyclic groups $H_{n,0},\ldots,H_{n,e_n}$ such that
\[
H_n \cong H_{n,0} \oplus \ldots \oplus H_{n,e_n}.
\]
Using theses decompositions of $H_n$ and $H_{n+1}$, the morphism $\varphi_{n+1,n}$ corresponds to a tuple $(\varphi_{n+1,n}^{(s,t)})_{s,t}$ of morphisms $\varphi_{n+1,n}^{(s,t)}\colon H_{n,s}\to H_{n+1,t}$, for $s=0,\ldots,e_n$ and $t=0,\ldots,e_{n+1}$.
Given $n\in\NN$ and $k\in\{0,\ldots,e_n\}$, set $B_{n,k}:=\Sigma$ if $H_{n,k}\cong\ZZ$, and set $B_{n,k}:=I_m$ if $H_{n,k}\cong\ZZ_m$.
We further define
\[
B_n := B_{n,0} \oplus \ldots \oplus B_{n,e_n}.
\]
Then $K_\ast(B_n)\cong (0,H_n)$.

Let $n\in\NN$, $s\in\{0,\ldots,e_n\}$ and $t\in\{0,\ldots,e_{n+1}\}$.
It follows from \cite[Lemma~13.2.1, p.222]{RorLarLau00KThy} that there exists a morphism $\psi_{n+1,n}^{(s,t)}\colon B_{n,s}\otimes\KK\to B_{n+1,t}\otimes\KK$ such that $K_1(\psi_{n+1,n}^{(s,t)})=\varphi_{n+1,n}^{(s,t)}$.
This induces a morphism $\psi_{n+1,n}\colon B_n\otimes\KK\to B_{n+1}\otimes\KK$ with $K_1(\psi_{n+1,n})=\varphi_{n+1,n}$.

Let $A_1$ be the inductive limit of the system $(B_n,\psi_{n+1,n})$.
Then
\[
K_1(A_1) \cong\varinjlim_n K_1(B_n) \cong\varinjlim_n H_n \cong G_1.
\]
Further, $K_0(B_n)\cong 0$ for each $n$, and therefore $K_0(A_1)\cong 0$.

Each $B_n$ is stable, nuclear and satisfies the UCT, whence $A_1$ has the same properties.
Moreover, each $B_n$ is homotopy symmetric.
By \cite[Theorem~3]{Dad93AsymHtpy}, homotopy symmetry passes to sequential inductive limits of separable \ca{s}.
Moreover, $A_1$ has a strong shape system since each $B_n$ is semiprojective.
\end{proof}

\begin{lma}
\label{prp:ShEquiv-EEquiv_for_HS}
Let $A$ and $B$ be separable, stable, homotopy symmetric \ca{s}.
Then $A$ and $B$ are shape equivalent if and only if they are $E$-equivalent.
\end{lma}
\begin{proof}
By \cite[Theorem~3.9]{Dad94ShapeThyAsymMor}, two separable \ca{s} are shape equivalent if and only if they are equivalent in the asymptotic homotopy category $\mathcal{A}$ studied in \cite{Dad94ShapeThyAsymMor}.
By \cite[Theorem~4.3]{DadLor94UnsuspendedEThy}, two stable, homotopy symmetric \ca{s} are equivalent in $\mathcal{A}$ if and only if they are $E$-equivalent.
\end{proof}


\begin{thm}
\label{prp:ssa_for_HS}
Let $A$ be a separable, stable, nuclear, homotopy symmetric \ca{} satisfying the UCT.
Assume that $K_0(A)$ is torsion-free.
Then $A$ has a strong shape system.
\end{thm}
\begin{proof}
Apply \autoref{prp:realizeKThy_sssHs} to obtain a stable, nuclear, homotopy symmetric \ca{} $B$ satisfying the UCT, with a strong shape system, and such that $K_\ast(A)\cong K_\ast(B)$.
We will show the following implications:
\[
K_\ast(A)\cong K_\ast(B)
\ \Rightarrow\ A\sim_{KK} B
\ \Rightarrow\ A\sim_E B
\ \Rightarrow\ A\sim_{\Sh} B,
\]
where $\sim_{KK}$, $\sim_E$ and $\sim_{\Sh}$ mean $KK$-equivalence, $E$-equivalence and shape equivalence, respectively.

The first implication follows from \cite[Corollary~23.10.2, p.241]{Bla98KThy}, which shows that two \ca{s} satisfying the UCT are $KK$-equivalent if (and only if) they have isomorphic $K$-groups.
The second implication follows from \cite[Theorem~25.6.3, p.278]{Bla98KThy}, which shows that $E$-theory and $KK$-theory agree for separable, nuclear \ca{s}.
The third implication follows from \autoref{prp:ShEquiv-EEquiv_for_HS}.

Hence, $A$ and $B$ are shape equivalent.
Since $B$ has a strong shape system, it follows from \autoref{prp:sss_shDomination} that $A$ does as well.
\end{proof}

\begin{rmks}
(1)
\autoref{prp:ssa_for_HS} still holds if the assumption that $A$ is nuclear and satisfies the UCT is replaced by the assumption that $A$ satisfies the UCT for $E$-theory as considered for example in \cite[25.7.5, p.281]{Bla98KThy}.

(2)
It is easy to construct \ca{s} to which \autoref{prp:ssa_for_HS} applies.
Let $B$ be a nuclear, stable, homotopy symmetric \ca{} satisfying the UCT.
By \cite[22.3.5(f), p.229]{Bla98KThy}, the class of nuclear \ca{s} satisfying the UCT is closed under tensor products.
Hence, for any nuclear \ca{} $A$ satisfying the UCT, the tensor product $A\otimes B$ is nuclear, stable, homotopy symmetric and satisfies the UCT.

Examples of homotopy symmetric \ca{s} include $C_0(X\setminus\{\ast\})$ for any connected, compact, metrizable space $X$;
see \cite{Dad93AsymHtpy} and \cite[Proposition~5.3]{DadLor94UnsuspendedEThy}.

(3)
It was recently shown by Gabe that a separable, nuclear \ca{} is homotopy symmetric if and only if its primitive ideal space contains no non-empty, compact, open subsets;
see \cite[Corollary~E]{Gab18arX:tracelessAFemb}.
\end{rmks}

\begin{cor}
\label{prp:sss_suspension_nuclear}
Let $A$ be a separable, nuclear \ca{} satisfying the UCT.
If $K_0(A)$ is torsion-free, then $\Sigma^2 A\otimes\KK$ has a strong shape system.
If $K_1(A)$ is torsion-free, then $\Sigma A\otimes\KK$ has a strong shape system.
\end{cor}

\begin{exa}
\label{exa:sss_HS_Kirchberg}
Recall from \autoref{exa:ShKirchberg} that every UCT-Kirchberg algebra has a strong shape system.
Let $A$ be a UCT-Kirchberg algebra, let $X$ be a connected, compact, metrizable space, and let $x\in X$.
If $K_0(C_0(X\setminus\{x\})\otimes A)$ is torsion-free, then $C_0(X\setminus\{x\})\otimes A\otimes\KK$ has a strong shape system.
\end{exa}

The restriction on the $K$-theory in \autoref{prp:realizeKThy_sssHs} and \autoref{prp:ssa_for_HS} arises since we do not know if there exist suitable building blocks realizing torsion in $K_0$.
In particular, we do not know the answer to the following question.

\begin{qst}
\label{qst:exist_sssHS}
Does $\Sigma I_n\otimes\KK$ have a strong shape system, for $n\geq 2$?
\end{qst}

\begin{rmk}
\label{rmk:exist_sssHS}
Let $n\geq 2$.
\autoref{qst:exist_sssHS} is closely related to \cite[Question~7.2]{DadLor94UnsuspendedEThy}, where Dadarlat and Loring ask if there exists a separable, nuclear, homotopy symmetric, \emph{semiprojective} \ca{} with $K$-theory $(\ZZ_n,0)$.

Note that $\Sigma I_n\otimes\KK$ is separable, stable, nuclear, homotopy symmetric and satisfies the UCT.
Hence, the argument in the proof of \autoref{prp:ssa_for_HS} shows that \autoref{qst:exist_sssHS} is equivalent to the following:
Does there exist \emph{any} separable, stable, nuclear, homotopy symmetric \ca{} $D$ satisfying the UCT, with strong shape system, and such that $K_\ast(D)\cong(\ZZ_n,0)$?

In particular, a positive answer to the question of Dadarlat and Loring (additionally assumed to satisfy the UCT) implies a positive answer to \autoref{qst:exist_sssHS}.
\end{rmk}


\providecommand{\bysame}{\leavevmode\hbox to3em{\hrulefill}\thinspace}
\providecommand{\noopsort}[1]{}
\providecommand{\mr}[1]{\href{http://www.ams.org/mathscinet-getitem?mr=#1}{MR~#1}}
\providecommand{\zbl}[1]{\href{http://www.zentralblatt-math.org/zmath/en/search/?q=an:#1}{Zbl~#1}}
\providecommand{\jfm}[1]{\href{http://www.emis.de/cgi-bin/JFM-item?#1}{JFM~#1}}
\providecommand{\arxiv}[1]{\href{http://www.arxiv.org/abs/#1}{arXiv~#1}}
\providecommand{\doi}[1]{\url{http://dx.doi.org/#1}}
\providecommand{\MR}{\relax\ifhmode\unskip\space\fi MR }
\providecommand{\MRhref}[2]{%
	\href{http://www.ams.org/mathscinet-getitem?mr=#1}{#2}
}
\providecommand{\href}[2]{#2}

\end{document}